\documentclass[12pt]{amsart}
\usepackage{}

\usepackage{amsmath}
\usepackage{amsfonts}
\usepackage{amssymb}
\usepackage[all,cmtip]{xy}           
\usepackage{bm}
\usepackage{bbm}
\usepackage{xcolor}
\usepackage{eucal}
\usepackage{bbding}
\usepackage{txfonts}
\usepackage{amscd}
\usepackage{makecell}
\usepackage[shortlabels]{enumitem}
\usepackage{ifpdf}
\ifpdf
  \usepackage[colorlinks,final,backref=page,hyperindex]{hyperref}
\else
  \usepackage[colorlinks,final,backref=page,hyperindex,hypertex]{hyperref}
\fi
\usepackage{tikz}
\usepackage[active]{srcltx}
\usepackage{caption}
\usepackage{comment}

\makeatletter

\newtheorem{thm}{Theorem}[section]
\newtheorem{lem}[thm]{Lemma}
\newtheorem{cor}[thm]{Corollary}
\newtheorem{pro}[thm]{Proposition}
\newtheorem{rmk}[thm]{Remark}
\newtheorem{defi}[thm]{Definition}

\theoremstyle{definition}

\newcommand{\be }{\begin{equation}}
\newcommand{\ee }{\end{equation}}

\newcommand{\g}{\mathfrak g}
\newcommand{\h}{\mathfrak h}

\newcommand{\huaA}{\mathcal{A}}

\newcommand{\huaW}{\mathcal{W}}

\newcommand{\huaC}{{\mathcal{C}}}

\newcommand{\frke}{\mathfrak e}

\newcommand{\dM}{\mathrm{d}}
\newcommand{\D}{\mathrm{D}}

\newcommand{\Id}{{\rm{Id}}}

\newcommand{\br}[1]{   [ \cdot,    \cdot  ]   }

\newcommand{\Hom}{\mathrm{Hom}}

\newcommand{\Harr}{\mathrm{Harr}}
\newcommand{\CE}{\mathrm{CE}}
\newcommand{\MC}{\mathrm{MC}}
\newcommand{\coker}{\mathrm{coker}}
\newcommand{\Prime}{\mathrm{Prime}}
\newcommand{\Gp}{\mathrm{Gp}}
\newcommand{\conil}{\mathrm{conil}}
\newcommand{\I}{\mathrm{I}}

\newcommand{\CDGCog}{\mathrm{\bf DGCCog}}
\newcommand{\cuLA}{\mathrm{\bf CuLA}}

\newcommand{\Cof}{\mathrm{Cof}}
\newcommand{\rlp}{\mathrm{rlp}}

\newcommand{\Coprod}{\mathrm{Coprod}}

\newcommand{\Der}{\mathrm{Der}}

\newcommand{\Set}{\mathrm{Set}}

\newcommand{\K}{k}

\begin{document}

\title[Global Koszul duality]{Global Koszul duality: differential graded cocommutative coalgebras and curved Lie algebras}

\author{Alexander Mallon}
\address{Department of Mathematics and Statistics, Lancaster university, Lancaster LA1 4YF, UK}
\email{a.mallon@lancaster.ac.uk}

\author{You Wang}
\address{Department of Mathematics, Jilin University, Changchun 130012, Jilin, China}
\email{wangyou888@jlu.edu.cn}

\vspace{-5mm}


\begin{abstract}
We give a combinatorial model structure to the category of, not necessarily conilpotent, differential graded (dg) cocommutative coalgebras and an $\infty$-category structure to the category of curved Lie algebras over an algebraically closed field of characteristic $0$. Further, we extend the Harrison and Chevally-Eilenberg functors between dg cocommutative conilpotent coalgebras and dg Lie algebras to these categories and show they form an equivalence of $\infty$-categories.
\end{abstract}

\keywords{Koszul duality, pseudocompact algebra, model category, Maurer-Cartan element, curved Lie algebra.}

\maketitle

\tableofcontents

\allowdisplaybreaks


\section{Introduction}
There are many manifestations of the Koszul duality between differential graded (dg) cocommutative coalgebras and dg Lie algebras, which exhibits an equivalence of the homotopy theory on their respective categories.
The first such equivalence was shown by Quillen \cite{Quillen}, which, in addition to introducing the concepts of model categories and Quillen equivalences between them, exhibited a Quillen equivalence between Lie algebras and cocommutative coalgebras. However, this came with strong grading restrictions. Later on, Hinich \cite{Hinich} generalised this result, removing the grading restriction. This was achieved by introducing a finer concept of weak equivalence on the category of cocommutative coalgebras in filtered quasi-isomorphisms.

Positselski \cite{Pos1} has constructed various forms Koszul duality between dg coassociative coalgebras and dg associative algebras, where the (co)algebras considered are not (co)augmented. The removal of augmentations from one side of the equivalence comes at the cost of adding curvature to the other.
A curved algebra, like a dg algebra, has an associated degree 1 derivation $d$, but is not subject to the condition $d^2 = 0$. Rather, there is a degree 2 element $h \in A^2$ such that $d^2(a) = [h,a] = ha - ah$ and $d(h)=0$. Curved coalgebras are defined dually with the curvature defined as degree 2 map $h:C\rightarrow k$ satisfying a dual condition.
A morphism between curved algebras $(f,b):A\rightarrow B$ comes equipped with a change of curvature $b\in B$. Curved morphisms between curved algebras must commute with the differential up to the change in curvature, that is, $fd(a) = df(a) + [b,f(a)]$.
This philosophy is taken in further generalisations of the Koszul duality of cocommutative coalgebras and Lie algebras, where the removal of coaugmentations on the category of cocommutative coalgebra side results in curvature on the category of Lie algebras.

In the above cases, the coalgebras considered are always conilpotent, a restriction that can be lifted in certain cases albeit by replacing the category of curved Lie algebras. Chuang, Lazarev and Mannan \cite{CLM} considers all (not necessarily conilpotent) coalgebras over an algebraically closed field, showing that the resulting category is Quillen equivalent to the category of formal coproducts of curved Lie algebras with strict morphisms between them.
This was achieved using the result that any coalgebra over an algebraically closed field splits as a coproduct of conilpotent coalgebras or, dually, the equivalence between the categories of pseudo-compact algebras and the category of formal products of local algebras.
Another generalisation is constructed by Maunder \cite{Maunder} wherein the category of marked curved Lie algebras is considered. These are curved Lie algebras with a distinguished non-empty set of degree $1$ elements called marked points and curved Lie algebra morphisms which preserve marked points. A curved Lie algebra can be twisted by a marked point $x$ by changing the differential to $d^x(a) := d(a) + [a,x]$. A twisted curved Lie algebra will usually also be curved. The result will be a dg Lie algebra (that is, with curvature $h = 0$) when the marked point is a Maurer Cartan (MC) element. These are degree one elements which satisfy $dx +\frac{1}{2}[x,x] = h$. In such a case, quasi-isomorphisms can be considered. A similar result is shown by Booth and Lazarev \cite{BooLaz} for the categories of curved associative algebras and not necessarily conilpotent curved coassociative coalgebras, where a Quillen equivalence given by an extended bar-cobar adjunction is given.

We  consider the category of curved Lie algebras, with curved morphisms. A map between two curved Lie algebras will be a weak equivalence if, for all MC elements, the resulting twisted map between dg Lie algebras is a quasi-isomorphism. Unlike the category of marked curved Lie algebras, this will also include Lie algebras with no MC elements. For example, the one dimensional Lie algebra $L\langle h\rangle$ generated by one element $h$ in degree 2 which also acts as the curvature. This is the initial object in the category of curved Lie algebras with strict morphisms. In this approach, MC elements in a curved Lie algebra act as basepoints do in a topological space with multiple path components.
For a curved Lie algebra $(X,d,h)$, an MC element $x \in X$ corresponds precisely to a curved morphism $(0,x):(0,0,0)\rightarrow (X,d,h)$ from the terminal curved Lie algebra. The twisting of $(X,d,h)$ by $x$ to give the dg Lie algebra $(X,d^x,0)$ corresponds to the choice of a path component in a topological space.

Unlike the above cases, we will show that the categories of dg commutative coalgebras and curved Lie algebras are equivalent on the level of $\infty$-categories.
These categories, with their respective classes of weak equivalences, each have an associated $\infty$-category in the sense of \cite{BarKan}, where relative categories give a model for $\infty$-categories. The $\infty$-categorical equivalence is induced by the adjunction $\Harr\dashv \check{\CE}$ of the Harrison functor and the extended Chevally-Eilenberg functor.
Further, since the $\infty$-category structure of cocommutative coalgebras comes from a combinatorial model category, we find that the category of curved Lie algebras is equivalent, as an $\infty$-category, to a combinatorial model category.

\subsection{Notation and conventions.}

In this paper, we work over an algebraically closed field $\K$. All vector spaces are over $\K$ and differential graded (dg) vector spaces are assumed to be cohomologically $\mathbb{Z}$-graded. We will denote cohomological gradings with superscripts. Given a graded vector space $V$, its suspension (or a shift) $\Sigma V$ is a graded vector space with $(\Sigma V)^i=V^{i-1}$. We also denote the shift   $\Sigma V$ by $V[1]$ and the inverse functor of $[1]$ by $[-1]$.

We denote the category of dg Lie algebras by $\bf DGLA$ and the category of curved Lie algebras with curved morphisms by $\bf \cuLA$. In addition, the category of curved Lie algebras with an adjoined initial object $\emptyset$ by $\bf CuLA_{\emptyset}$. All coalgebras are assumed to be differential graded unless stated otherwise. The categories of coaugmented and conilpotent cocommutative coalgebras are denoted $\CDGCog^{coaug}$ and $\CDGCog^{conil}$ respectively. The category of cocommutative coalgebras, which are not necessarily coaugmented or conilpotent, is denoted $\CDGCog$. For results and terminology about coalgebras, we refer the reader to Positselski's recent survey \cite{Pos1}.
\vspace{2mm}

\subsection{Outline of the paper}

In section 2, we recall the definitions of the Harrison and Chevally-Eilenberg functors which form an adjoint pair between the categories of coaugmented cocommutative coalgebras and dg Lie algebras. We then extend this adjunction to the categories of (not necessarily conilpotent) cocommutative coalgebras and curved Lie algebras. In both situation, we prove adjointness by showing the Harrison and Chevally-Eilenberg complexes are representing objects for the functor of Maurer-Cartan elements. We then show this new adjunction slices to an adjunction between dg Lie algebras and coaugmented cocommutative coalgebras.

In section 3, we introduce a model category structure on ${\CDGCog}$. This model structure is induced from the combinatorial model structure on ${\CDGCog^{conil}}$ and the equivalence ${\CDGCog} \cong {\bf CoProd(\CDGCog^{conil})}$. We show that the induced model structure on the category of formal coproducts of a combinatorial model category is itself combinatorial.

In section 4, we introduce the weak equivalences on $\cuLA_\emptyset$ which endows it with an $\infty$-category structure. We then show that $\Harr$ and $\check{\CE}$ are mutual $\infty$-categorical equivalences between the categories $\cuLA_\emptyset$ and $\CDGCog$ by showing that the components of the unit and counit of the adjunction are weak equivalences in their respective categories. This gives the result that the corresponding homotopy categories are equivalent.
\vspace{2mm}

\subsection{Acknowledgements.}
We would like to thank Andrey Lazarev for numerous helpful and fruitful discussions during the writing of this paper, and for his corrections and remarks on earlier drafts.

\section{Extended Chevalley-Eilenberg functors}

The classical Harrison and Chevalley-Eilenberg functors give an adjunction between the categories of (coaugmented) conilpotent cocommutative coalgebras and dg Lie algebras in \cite{Hinich,Quillen}. In this section, we describe a non-conilpotent version of this adjunction, before extending it to curved Lie algebras. The Harrison functor remains the same, but the Chevalley-Eilenberg functor $\CE$ is replaced with the extended Chevalley-Eilenberg functor $\check{\CE}$. The main difference in the construction of $\check{\CE}$ is that we must replace the cofree conilpotent cocommutative coalgebra with the cofree cocommutative coalgebra (a much larger coalgebra).

\subsection{Commutative pseudocompact algebras}

Let $\{V_i\}$ be a cofiltered system of finite dimensional $k$-vector spaces. The inverse limit $\varprojlim\limits_i V_i$ can be endowed with the inverse limit topology, where each finite-dimensional vector space is endowed with the discrete topology. A topological vector space $V$ is said to be pseudocompact if $V$ is isomorphic to $\varprojlim\limits_i V_i$. Every vector space $V$ can be viewed as the filtered colimit of its finite-dimensional subspaces, making its dual $V^*$ naturally pseudocompact. Given pseudocompact vector spaces $V$ and $W$, the space of morphisms $\Hom(V,W)$ consists of continuous linear maps and the tensor product between $V$ and $W$ is the completed tensor product $V\hat{\otimes} W$. If $V=\varprojlim\limits_i V_i$ is pseudocompact and $W$ is discrete, then their tensor product is defined by
$V\otimes W=\varprojlim\limits_i V_i \otimes W$, which is neither discrete nor pseudocompact in general.
Similarly, if $\{A_i\}$ is a cofiltered system of finite-dimensional $k$-algebras, its colimit can also be equipped with the inverse limit topology. Such a topological algebra is likewise referred to as pseudocompact. In particular, a topological algebra is pseudocompact if and only if it is complete, Hausdorff, and has a basis consisting of ideals with finite codimension.

Suppose $C$ is a coalgebra. We can express $C$ as the union of its finite-dimensional subcoalgebras $\{C_i\}$, according to a well-known result by Sweedler. Then, its linear dual $C^* \cong \varprojlim C_i$ is naturally a pseudocompact algebra. If $A$ is a pseudocompact algebra, its topological dual $A^*$ is a coalgebra. By convention, the dual of a pseudocompact algebra always refers to its topological dual, hence we have $C^{**} \cong C$ and $A^{**} \cong A$. Furthermore, the linear and topological duals establish a contravariant equivalence between the categories of pseudocompact (commutative) algebras and (cocommutative) coalgebras. The category of (graded) pseudocompact algebras is denoted by {\bf pcAlg}, the category of pseudocompact dg (commutative) algebras is denoted by {\bf pcDGA} (or {\bf pcDGCAlg}) and the category of augmented pseudocompact dg (commutative) algebras is denoted by ${\bf pcDGA^{\bf aug}}$ (or ${\bf pcDGCAlg^{\bf aug}}$). The duality between pseudocompact algebras and coalgebras generalise to the graded, differential graded, and (co)augmented cases.

Given an algebra $A$, its pseudocompact completion $\check{A}$ is the projective limit of the inverse system of quotients of $A$ by its cofinite dimensional ideals. The pseudocompact completion defines a functor from $ \bf{Alg} \to \bf{pcAlg} $, which is left adjoint to the functor $\bf{pcAlg} \to \bf{Alg}$ that forgets the topology.

\begin{defi}
Let $V$ be a pseudocompact graded vector space. If $V$ is finite-dimensional, its commutative pseudocompact tensor algebra $\check{S}V$ is the pseudocompact completion of the commutative tensor algebra $SV$. For a general pseudocompact vector space $V = \varprojlim\limits_i V_i $, its commutative pseudocompact tensor algebra is defined as
$$
\check{S} V := \varprojlim\limits_i \check{S} V_i.
$$
\end{defi}

\begin{pro}\label{checkSV}
Let $V$ be a pseudocompact graded vector space.
\begin{enumerate}
    \item The commutative pseudocompact tensor algebra $\check{S}V$ is the free commutative pseudocompact algebra on $V$. That is, it is left adjoint to the functor that sends a commutative pseudocompact algebra to its underlying pseudocompact vector space. This means we have the following bijections for all pseudocompact vector spaces $V$ and commutative pseudocompact algebras $A$:
    $$\Hom_{\bf pcDGCAlg}(\check{S}V, A) \cong \Hom_{\bf pcVect}(V, A).$$
    \item For any pseudocompact $\check{S}V-\check{S}V-$bimodule $M$, there is a bijection

    $$\Der(\check{S}V,M)\cong \Hom_{\bf pcVect}(V,M).$$

    In particular, when $M$ is also $\check{S}V$, we have

    $$\Der(\check{S}V)\cong \Hom_{\bf pcVect}(V,\check{S}V).$$

\end{enumerate}
\end{pro}

\begin{proof}
\begin{enumerate}
  \item If $V$ is finite-dimensional, then $V$ is discrete and $$\Hom_{\bf dg vect}(V, A) \cong \Hom_{\bf CDGA}(SV, A),$$which equals $\Hom_{\bf pcDGCAlg}(\check{S}V, A)$ as pseudocompact completion is left adjoint to the forgetful functor that forgets the topology. More generally, writing $V=\varprojlim\limits_i V_i$ and $A=\varprojlim\limits_j A_j$ with $V_i,A_j$ finite-dimensional for all $i$ and $j$ respectively, we have
        \begin{eqnarray*}
        \Hom_{\bf pcDGCAlg}(\check{S}V, A) &\cong& \varprojlim\limits_j \Hom_{\bf pcDGCAlg} (\check{S}V, A_j)\\
        &\cong& \varprojlim\limits_j \varinjlim\limits_i \Hom(\check{S}V_i, A_j)\\
        &\cong& \varprojlim\limits_j \varinjlim\limits_i \Hom(V_i, A_j)\\
        &\cong& \Hom_{\bf pcVect}(V, A).
        \end{eqnarray*}
        Here, the second bijection holds as finite-dimensional commutative algebras are compact in
        {\bf pcDGCAlg}, that is, the functor $\Hom_{\bf pcDGCAlg}(-,A)$ commutes with filtered colimits.
  \item Given a pseudocompact dg commutative algebra $A$ and an $A-A-$bimodule $M$, consider the pseudocompact dg commutative algebra $A\oplus M$ with the following multiplication:
      $$ (a,m)\cdot(b,n)=(ab,an+mb), \quad \forall a,b\in A,m,n\in M.$$
      Let $p:A\oplus M\to A$ be the natural projection. Then there exists a bijection:
      $\Der(A,M)\cong \{f\in\Hom_{\bf pcDGCAlg} (A,A\oplus M) |p\circ f=\Id_A \}$.
      Setting $A=\check{S}V$ and using (1), we have
      $$ \Der(\check{S}V,M)\cong \{f\in\Hom_{\bf pcDGCAlg} (\check{S}V,\check{S}V \oplus M) |p\circ f=\Id_A \} \cong \Hom_{\bf pcVect}(V,M).  $$
\end{enumerate}
\end{proof}

\begin{rmk}
The commutative pseudocompact tensor algebra $\check{S}V$ is the $\K-$linear dual to the Sweedler cofree cocommutative coalgebra on the discrete vector space $V^*$.
\end{rmk}

\begin{defi}\label{first-version}
We define a pair of functors as follows:
$$ \Harr: {\CDGCog}^{\bf coaug}\rightleftarrows {\bf DGLA}:\check{\CE}.$$
The Harrison functor $\Harr$ sends a coaugmented cocommutative coalgebra $(C,\Delta_C,\delta_C)$ to the dg Lie algebra
$$ \Harr(C):= \Big( \mathcal{L}(\Sigma^{-1} \bar{C}),~[\cdot,\cdot],\dM_\Harr\Big),$$
where $\mathcal{L}(\Sigma^{-1} \bar{C})$ is the free Lie algebra spanned by the primitive elements of the tensor algebra $T(\Sigma^{-1} \bar{C})$ where $\bar{C}$ is the cokernel of the coaugmentation. The Lie bracket $[\cdot,\cdot]$ is then commutator of the concatenation product on $T(\Sigma^{-1} \bar{C})$. The differential $\dM_\Harr$ on  $\Harr(C)$ is the unique extension of $d_1+d_2$ by the graded Leibniz rule, where $\dM_1:\Sigma^{-1} \bar{C} \to \Sigma^{-1} \bar{C}$ and $\dM_2:\Sigma^{-1} \bar{C} \to \Sigma^{-1} \bar{C}\otimes \Sigma^{-1} \bar{C}$ are induced by the differential $\delta_C$ and comultiplication $\Delta_C$ respectively.

The extended Chevalley-Eilenberg Functor $\check{\CE}$ sends a dg Lie algebra $(\g,[\cdot,\cdot]_{\g},\dM_{\g})$ to the coaugmented cocommutative coalgebra $\check{\CE}(\g)$, which is the topological dual of the following dg commutative pseudocompact algebra
$ \Big( \check{S}(\Sigma^{-1} \g^*),\hat{\odot},\dM^*_{\check{\CE}}\Big)$.
Here, $\check{S}(\Sigma^{-1} \g^*)$ is the pseudocompact completion of $S(\Sigma^{-1} \g^*)$ and the multiplication $\hat{\odot}$ is the completion of the graded commutative multiplication ${\odot}$ on $S(\Sigma^{-1} \g^*)$. We define the differential on $\check{S}(\Sigma^{-1} \g^*)$ as follows. Let $\delta_1:\Sigma^{-1} \g^* \to \Sigma^{-1} \g^*$ and $\delta_2:\Sigma^{-1} \g^* \to \Sigma^{-1} \g^*\otimes \Sigma^{-1} \g^*$ be induced by the differential $\dM_{\g}$ and Lie bracket $[\cdot,\cdot]_{\g}$ respectively. For a pseudocompact vector space $V$, consider the semi-completed commutative tensor algebra $S'(V)=\oplus_{n\geqslant 1} V^{\hat{\odot} n}$, which has a topology neither pseudocompact nor discrete, and has a property $\Hom(S'(V),B)\cong \Hom(V,B)$ for any commutative pseudocompact algebra $B$. More details can be seen in \cite[Lemma 4.5]{Guan}. By Proposition \ref{checkSV}, the identity on $\check{S}(\Sigma^{-1} \g^*)$ induces a map $i:S'(\Sigma^{-1} \g^*)\to \check{S}(\Sigma^{-1} \g^*)$, we then have a map:
$$ i\circ (\delta_1+\delta_2):\Sigma^{-1} \g^* \to S'(\Sigma^{-1} \g^*)\to \check{S}(\Sigma^{-1} \g^*), $$
which can be extended to the differential $\dM^*_{\check{\CE}}$ on $\check{S}(\Sigma^{-1} \g^*)$ by the graded Leibniz rule. Taking the topological dual of this algebra yields the cocommutative coalgebra $(\check{CE(\g)},\Delta,\dM_{\check{\CE}})$.
\end{defi}

\subsection{Maurer-Cartan elements and representability}
We now show that the functors defined above form an adjuction $\Harr \dashv \check{\CE}$. This is done by showing that both functors represent the functor of Maurer-Cartan elements. We now recall the definition of Maurer-Cartan (MC) elements.

\begin{defi}
    Let $(\g,[\cdot,\cdot],\dM)$ be a dg Lie algebra. A {\bf Maurer-Cartan element} in $\g$ is an element $x\in \g$ of degree $1$ such that $\dM x+\frac{1}{2}[x,x]=0.$ The set of all Maurer-Cartan elements in $\g$ is denoted by $\MC(\g)$. It is immediate that dg algebra morphisms preserve MC elements. Hence, MC elements define a functor $\MC:\bf{DGLA}\rightarrow \bf{Set}$.
\end{defi}

We now recall that the graded vector space of linear maps between a dg Lie algebra and a cocommutative coalgebra can itself be given the structure of a dg Lie algebra.

\begin{pro}
Let $(C,\Delta_C,\delta_C)$ be a cocommutative coalgebra and $(\g,[\cdot,\cdot]_{\g},\dM_{\g})$ be a dg Lie algebra. Then $(\Hom(C,\g),[\cdot,\cdot]_{\ast},\partial)$ is a dg Lie algebra, where $[\cdot,\cdot]_{\ast}$ is the convolution product given by
$$[f,g]_{\ast}(x)= (-1)^{|g||x_{(1)}|}  [f(x_{(1)}), g(x_{(2)})]_{\g},$$
Where, $\forall x\in C$, the coproduct $\Delta_C(x)=x_{(1)}\otimes x_{(2)}$ is written in Sweedler notation with the summation left implicit.
The differential $\partial$ is defined by
$$ \partial(f)=\dM_{\g} f-(-1)^{|f|} f\delta_C. $$
\end{pro}

\begin{proof}
$(\Hom(C,\g),[\cdot,\cdot]_{\ast})$ is a graded Lie algebra, since $(C,\Delta_C,\delta_C)$ is a cocommutative coalgebra and $(\g,[\cdot,\cdot]_{\g},\dM_{\g})$ is a dg Lie algebra. Moreover,
by direct calculation, we have
\begin{eqnarray*}
\partial([f,g]_{\ast})
&=& \dM_{\g} [f,g]_{\ast}-(-1)^{|f+g|} [f,g]_{\ast} \delta_C\\
&=& [\dM_{\g} f,g]_{\ast}+(-1)^{|f|}[f,\dM_{\g} g]_{\ast}-(-1)^{|f|}[f\delta_C,g]_{\ast}-(-1)^{|f+g|}[f,g\delta_C]_{\ast}\\
&=& [\partial f,g]_{\ast}+(-1)^{|f|}[f,\partial g]_{\ast},
\end{eqnarray*}
and
\begin{eqnarray*}
\partial^2([f,g]_{\ast})
&=& \partial([\partial f,g]_{\ast}+(-1)^{|f|}[f,\partial g]_{\ast})\\
&=& [\partial^2 f,g]_{\ast}+(-1)^{|f|+1}[\partial f,\partial g]+(-1)^{|f|}[\partial f,\partial g]+[f,\partial^2 g]\\
&=& 0.
\end{eqnarray*}
So $\partial$ is a square-zero linear map of degree $1$ satisfying the graded Leibniz rule. Hence, we have shown $(\Hom(C,\g),[\cdot,\cdot]_{\ast},\partial)$ is a dg Lie algebra.
\end{proof}

\begin{defi}
    For any cocommutative coalgebra $(C,\Delta_C,\delta_C)$ and any dg Lie algebra $(\g,[\cdot,\cdot]_{\g},\dM_{\g})$, we define $\MC(C,\g):=\MC(\Hom(C,\g))$. This defines a bifunctor $$\MC:\CDGCog^{op}\times\bf{DGLA}\rightarrow \bf{Set}.$$
\end{defi}

We now show that $\Harr$ and $\check{CE}$ form an adjunction. We do this by showing that $\Harr(C)$ and $\check{CE}(\mathfrak{g})$ are representing objects for the functors $\MC(\bar{C},-)$ and $\MC(-,\mathfrak{g})$ respectively.

\begin{pro}\label{representable}
There is an adjunction between the categories of coaugmented cocommutative coalgebras and dg Lie algebras given by $\Harr \dashv \check{\CE}$.
\end{pro}

\begin{proof}
Let $C$ be a coaugmented cocommutative coalgebra and $\g$ be a dg Lie algebra. It is enough to show there are natural bijections
$$ \Hom_{\bf DGLA}(\Harr(C),\g)\cong \MC(\bar{C},\g) \cong \Hom_{\CDGCog^{coaug}}(C,\check{\CE}(\g)).$$

Forgetting the differential, any morphism of coaugmented cocommutative coalgebras $f:C \to \check{\CE}(\g)$ is dual to the map of augmented commutative pseudocompact algebras $f^{\ast}:\check{S}(\Sigma^{-1} \g^*) \to C^*$, which is equivalent to a linear map $\Sigma^{-1} \g^* \to \bar{C^*}$ by Proposition \ref{checkSV}. The linear map $\Sigma^{-1} \g^* \to \bar{C^*}$ is also dual to the map $\bar{C} \to \Sigma \g$, which can be viewed as a degree $1$ element in $\Hom(\bar{C},\g)$. Now consider the morphism of dg Lie algebras $\Psi:\Harr(C)\to \g$, which is in one-to-one correspondence with the linear map $\bar{\psi}:\Sigma^{-1} \bar{C}\to \g$ by the universal property of free Lie algebra $\mathcal{L}(\Sigma^{-1} \bar{C})$. It is also equivalent to a degree $1$ element in $\psi\in \Hom(\bar{C},\g)$. Moreover, if $\Psi:\Harr(C)\to \g$ commutes with the differential, i.e. $\dM_{\g}\circ \Psi=\Psi \circ (\dM_1+\dM_2),$ then, for all $x\in C$, we have
\begin{eqnarray*}
&&\dM_{\g}\circ\Psi(s^{-1} x)-\Psi\circ (\dM_1+\dM_2)(s^{-1} x)\\
&=& \dM_{\g}\bar{\psi}(s^{-1} x)-\bar{\psi}(\dM_1+\dM_2)(s^{-1} x)\\
&=& \dM_{\g}\psi(x)+\psi \delta_{C} (x)+ \frac{1}{2}[\psi(x_{(1)}), \psi(x_{(2)})]_{\g}\\
&=& (\partial \psi+\frac{1}{2} [\psi,\psi]_{\ast})(x)\\
&=& 0,
\end{eqnarray*}
which shows that $\psi\in \MC(\bar{C},\g)$. The bijection $\Hom_{\bf DGLA}(\Harr(C),\g)\cong \MC(\bar{C},\g)$ follows similarly.
\end{proof}

\subsection{Curved Lie algebras and the corresponding extended Chevalley-Eilenberg Functors}

In this subsection, we first recall the definition of curved Lie algebras. More details can be seen in \cite{CLM,Maunder}.

\begin{defi}
A {\bf curved Lie algebra} $(\g,\dM_{\g},\omega_{\g})$ is a graded Lie algebra $\g$ equipped with a degree $1$ derivation $\dM_{\g}$ and a degree $2$ element $\omega_{\g}\in \g^2$ such that $\dM_{\g}(\omega_{\g})=0$ and $\dM^2_{\g}(x)=[\omega_{\g},x]_{\g}$ for all $x\in \g.$ $\omega_{\g}$ is called the {\bf curvature} of $\g$. Note that a dg Lie algebra is precisely a curved Lie algebra with zero curvature.
\end{defi}

\begin{defi}\label{curved-morph}
A {\bf curved morphism} of curved Lie algebras is a pair
$$(f,a):(\g,\dM_{\g},\omega_{\g}) \to (\h,\dM_{\h},\omega_{\h}),$$
 where $f:\g \to \h$ is a graded Lie algebra morphism and a degree $1$ element $a \in \h^{1}$, called the {\bf change of curvature}, satisfying the following:
\begin{itemize}
\item[$\bullet$] $f(\dM_{\g} x)=\dM_{\h}(fx)+[a,fx]_{\g};$
\item[$\bullet$] $f(\omega_{\g})=\omega_{\h}+\dM_{\h} a+\frac{1}{2} [a,a]_{\h}.$
\end{itemize}
\end{defi}
We can compose curved morphisms by setting $(g,b)\circ (f,a)=(gf,b+g(a))$. Further, the morphism $(\Id,0)$ acts as a unit with respect to composition. We can therefore define the category $\cuLA$ of curved Lie algebras with curved morphisms.

A morphism $(f,a)$ is called uncurved (or strict) if $a=0$. Any curved morphism $(f,a)$ decomposes as the composition $(\Id,a)\circ(f,0)$ of an uncurved morphism with a curved isomorphism. Strict morphisms between dg Lie algebras correspond exactly to morphisms in $\bf DGLA$. As there can be many curved morphisms between two dg Lie algebras, the category $\bf DGLA$ does not embed fully into $\cuLA$.

\begin{rmk}
    The category of curved Lie algebras with strict morphisms can also be considered. However, by $\cuLA$, we will always be referring to curved Lie algebras with curved morphisms.
\end{rmk}

We now give the corresponding definition of Maurer-Cartan elements in curved Lie Algebras.

\begin{defi}\label{MC}
    A Maurer-Cartan element of a curved Lie algebra $(\g,\dM_{\g},\omega_{\g})$  is an element $a\in \g$ such that $\omega_{\g}+\dM_{\g}x+\frac{1}{2} [x,x]_{\g}=0.$
\end{defi}

It follows from definitions \ref{curved-morph} and \ref{MC} that an MC element of a curved Lie algebra $(\g,\dM,\omega))$ is the same as the change of curvature of the inclusion of the zero Lie algebra into $\g$. That is, there is a natural bijection $\MC(\g) \cong \Hom_{\cuLA}((0,0,0)),(\g,\dM,\omega))$. Note that $(0,0,0)$ is the terminal object of $\cuLA$, as such, MC elements act like the points of a curved Lie algebra.

Given an MC element $x$ of a curved Lie algebra $(\g,\dM,\omega)$, we can define the twisted Lie algebra $(\g^x,\dM^x)$, whose underlying Lie algebra is $\g$ and whose differential is defined by $\dM^x(y)=\dM y+[x,y]$. By direct calculation, we can see that $(\dM^x)^2=0$ and $\dM^x$ is a graded derivation, so $(\g^x,\dM^x)$ is a dg Lie algebra. Moreover, $(\Id,x):(\g^x,\dM^x,0)\to (\g,\dM,\omega)$ is a curved isomorphism. Thus, if a curved Lie algebra $(\g,\dM,\omega)$ admits a Maurer-Cartan element $x$, then it is curved isomorphic to the dg Lie algebra $\g^x$ twisted by $x$.
Similarly, if two curved Lie algebras $(\g,\dM_{\g},\omega_{\g})$ and $(\h,\dM_{\h},\omega_{\h})$ admit Maurer-Cartan elements, then for any curved morphism $(f,\alpha):(\g,\dM_{\g},\omega_{\g})\to (\h,\dM_{\h},\omega_{\h})$, we can obtain a morphism of dg Lie algebras $(\Id,y)(f,\alpha)(\Id,x):(\g^x,\dM^x_{\g})\to (\h^y,\dM^y_{\h}),$ for all $x\in \MC(\g)$ and $y\in \MC(\h)$.

Curved morphisms of curved Lie algebras induce the map between their Maurer-Cartan elements sets.
\begin{pro}\label{map-in-MCset}
Let $(f,a):(\g,\dM_{\g},\omega_{\g})\to (\h,\dM_{\h},\omega_{\h})$ be a curved morphism. Then $(f,a)$ induces $\tilde{a}:\MC(\g) \to \MC(\h)$ which sends $x\in \MC(\g)$ to $f(x)+a \in \MC(\h).$
\end{pro}

\begin{proof}
We only need to show that $f(x)+a \in \MC(\h)$. Actually, we have
\begin{eqnarray*}
&&\omega_{\h}+\dM_{\h}(f(x)+a)+\frac{1}{2}[f(x)+a,f(x)+a]_{\h}\\
&=&f(\omega_{\g})-\frac{1}{2}[a,a]_{\h}-\dM_{\h}a+f(\dM_{\g}x)-[a,f(x)]_{\h}+\dM_{\h}a+\frac{1}{2}[f(x),f(x)]_{\h}
+\frac{1}{2}[a,a]_{\h}+[a,f(x)]_{\h}\\
&=& f(\omega_{\g}+\dM_{\g}x+\frac{1}{2}[x,x]_{\g})\\
&=& 0,
\end{eqnarray*}
Hence $f(x)+a \in \MC(\h)$ as required.
\end{proof}

\begin{rmk}
    This result also follows directly from the functoriality of $\Hom_{\cuLA}(0,-)$.   
\end{rmk}

We now extend the adjoint functors $(\Harr,\check{\CE})$ to the categories $\bf{CDGCog}$ of (non-conilpotent) cocommutative coalgebras and curved Lie algebras. We then show that $(\Harr,\check{\CE})$ still form an adjunction.

\begin{defi}
We define a pair of functors as follows:
$$ \Harr: {\CDGCog}\rightleftarrows {\cuLA}:\check{\CE}.$$
The Harrison functor sends a cocommutative coalgebra $(C,\Delta_C,\delta_C)$ to the curved Lie algebra
$$ \Harr(C):= \Big( \mathcal{L}(\Sigma^{-1} \bar{C}),~[\cdot,\cdot],\dM_{\Harr},\omega_{\Harr} \Big).$$
The differences compared to the former are the differential $\dM_{\Harr}$ and the curvature $\omega_{\Harr}$. Since $C$ is not coaugmented, we can choose a fake coaugmentation $\varepsilon:k \to C$, then we have $C\cong k \oplus \bar{C}$ as vector spaces and $\bar{C}:=\coker \varepsilon$. The differential $\dM_{\Harr}$ is induced by the restriction maps $\Delta_{\bar{C}}:\bar{C}\to \bar{C}\otimes \bar{C}$ and $\delta_{\bar{C}}:\bar{C}\to \bar{C}$, similar to the uncurved case. And the curvature $\omega_{\Harr}$ is induced by the restriction maps $\Delta_{k}:k\to \bar{C}\otimes \bar{C}$ and $\delta_{k}:k\to \bar{C}$. That means, for all $c\in \bar{C}$, we have
\begin{eqnarray*}
\dM_{\Harr}(s^{-1}c)&=&(-1)^{|c_{(1)}|+1} s^{-1}c_{(1)}\otimes s^{-1}c_{(2)}-s^{-1} \delta_{\bar{C}}(c);\\
\omega_{\Harr}&=&(-1)^{|\Delta_{(1)}|+1} s^{-1}\Delta_{(1)}\otimes s^{-1}\Delta_{(2)}-s^{-1} \delta_{k}(1_{k}),
\end{eqnarray*}
where $\Delta_{\bar{C}}(c)=c_{(1)}\otimes c_{(2)},\Delta_{k}(1_{k})=\Delta_{(1)}\otimes \Delta_{(2)}.$ Then we can check that $\dM_{\Harr}(\omega_{\Harr})=0$ and $\dM_{\Harr}^2 (s^{-1} x)=[\omega_{\Harr},s^{-1} x],$ for all $s^{-1} x \in \Sigma^{-1}\bar{C}$. Note that $\omega_{\Harr}$ is a primitive element of the bialgebra $(T(\Sigma^{-1}\bar{C}),\otimes,\Delta^{\rm cosh})$, i.e. $\Delta^{\rm cosh}(\omega_{\Harr})=\omega_{\Harr}\otimes 1_{\K}+1_{\K}\otimes \omega_{\Harr}$. Thus, $\omega_{\Harr}\in \mathcal{L}(\Sigma^{-1} \bar{C})=\Prime T(\Sigma^{-1}\bar{C}).$ Therefore,
$( \mathcal{L}(\Sigma^{-1} \bar{C}),\dM_{\Harr},\omega_{\Harr})$ is a curved Lie algebra.

The extended Chevalley-Eilenberg Functor $\check{\CE}$ sends a curved Lie algebra $(\g,[\cdot,\cdot]_{\g},\dM_{\g},\omega_{\g})$ to a cocommutative dg coalgebra $\check{\CE}(\g)$, which is the topological dual of the following commutative pseudocompact dg algebra
$ \Big( \check{S}(\Sigma^{-1} \g^*),\hat{\odot},i\circ (\delta_1+\delta_2+\delta_3)\Big).$
The only difference compared to the former is the differential adding a map $\delta_3:\Sigma^{-1} \g^* \to \K$ induced by $\omega_{\g}^*:({\g}^*)^{-2}\to \K$. We only need to show that $(\delta_1+\delta_2+\delta_3)^2=0$. Actually, the equation $(\delta_2)^2=0$ holds because of the graded Jacobi Identity of $[\cdot,\cdot]_{\g}$ and the equation $[\delta_1,\delta_2]=0$ expresses the fact that $\dM_{\g}$ is a derivation. The equation $[\delta_1,\delta_3]=0$ is equivalent to the condition $\dM_{\g}(\omega_{\g})=0$ and the equation $(\delta_1)^2+[\delta_2,\delta_3]=0$ is
equivalent to the condition $(\dM_{\g})^2=[\omega_{\g},\cdot]_{\g}$. Therefore, $(\delta_1+\delta_2+\delta_3)^2=0$, which implies that $\delta_1+\delta_2+\delta_3$ is a differential. Then we can extend it to a derivation $i\circ (\delta_1+\delta_2+\delta_3)$ on $\check{S}(\Sigma^{-1} \g^*)$.
\end{defi}

\begin{rmk}
Since any cocommutative coalgebra can be viewed as a coproduct of conilpotent cocommutative coalgebras, i.e. $C\cong \coprod\limits_{\alpha} C_{\alpha}$, we can obtain that ${\CDGCog}$ is equivalent to the formal coproduct category of conilpotent cocommutative coalgebras ${\bf CoProd(\CDGCog^{conil})}$. Then we can define the adjoint functors $(\Harr,\check{\CE})$ between
${\bf CoProd(\CDGCog^{conil})}$ and $\bf cuLA$ as follows:
\begin{eqnarray}
\label{Harr1}\Harr \left(\coprod\limits_{\alpha} C_{\alpha} \right)&=&\coprod\limits_{\alpha} \Harr(C_{\alpha});\\
\label{checkCE1} \check{\CE}(\g)&=&\coprod\limits_{x\in \g^{-1}} \CE(\g^x).
\end{eqnarray}
\end{rmk}

Similar to the uncurved case, we can also show that the functors $\Harr$ and $\check{\CE}$ still form an adjoint pair between commutative coalgebras and curved Lie algebras.

\begin{pro}\label{representable2}
There is an adjunction between the categories of cocommutative coalgebras and curved Lie algebras given by $\Harr \dashv \check{\CE}$.
\end{pro}

Recall that if $\mathcal{C}$ is a category and $c\in \mathcal{C}$ an object, then the undercategory $\mathcal{C}_{c/}$ is the category whose objects are maps $c\to x$ and whose morphisms are commutative triangles.

\begin{lem}\label{DGLA-cuLA}
There is a categorical equivalence ${\bf DGLA} \cong {\cuLA_{0/} }$.
\end{lem}

\begin{proof}
Every dg Lie algebra $\g$ admits a morphism $0\to \g$ of dg Lie algebras, which is also a morphism of curved Lie algebras, so the natural inclusion ${\bf DGLA} \to {\cuLA}$ has image contained in ${\cuLA_{0/}}$. Let $i:{\bf DGLA} \to {\cuLA_{0/} }$ be the natural inclusion. We show that $i$ is a categorical equivalence. A map $0\to \g$ of curved Lie algebra is the data of a Maurer-Cartan element $x\in \g$. In this case, $\g$ is curved isomorphic to dg Lie algebra $\g^x$, which implies that $i$ is essentially surjective. To see that it is fully faithful, let $\g$ and $\h$ be two Lie algebras. Then the morphism of dg Lie algebras $f:\g \to \h$ corresponds to the map $\g \to \h$ in ${\cuLA_{0/} }$ given by a pair $(f,0)$, which is a curved Lie morphism. So $i$ is fully faithful and essentially surjective, which implies ${\bf DGLA}$ and ${\cuLA_{0/} }$ are categorical equivalent.
\end{proof}

Suppose that $L:\mathcal{C} \rightleftarrows \mathcal{D}:R$ is an adjunction. Pick $d\in \mathcal{D}$ and put $c:=Rd$. If the counit $d\to Lc$ is an isomorphism in $\mathcal{D}$, then the above adjunction can slice to an adjunction of the undercategories $L:\mathcal{C}_{c/} \rightleftarrows \mathcal{D}_{d/}:R$. If $\mathcal{C}$ and $\mathcal{D}$ are ${\CDGCog}$ and ${\cuLA}$ respectively, and $d=0$, the above counit condition is satisfied and we can obtain a sliced adjunction
$$\Harr:{\CDGCog_{k/}}\rightleftarrows {\cuLA_{0/}}:\check{\CE}.$$

\begin{pro}
The Harrison-extended Chevalley-Eilenberg adjunction
$$ \Harr: {\CDGCog}\rightleftarrows {\cuLA}:\check{\CE}$$
slices to an adjunction
$$ \Harr: {\CDGCog}^{\bf coaug}\rightleftarrows {\bf DGLA}:\check{\CE}.$$
\end{pro}

\begin{proof}
This follows directly from Lemma \ref{DGLA-cuLA} and the fact that ${\CDGCog_{k/}} \cong {\CDGCog}^{\bf coaug}$.
\end{proof}

\section{The Model Structure on cocommutative coalgebras}

\subsection{Model category structure on cocommutative coalgebras}

First we recall the model structure on the category of conilpotent cocommutative coalgebras. Then we show that this category is actually combinatorial.

\begin{defi}{\rm(\cite{Hinich})}
A morphism of conilpotent cocommutative coalgebras $f:C \to D$ in ${\CDGCog^{conil}}$ is called
\begin{itemize}
\item[$\bullet$] a weak equivalence if and only if $\Harr(f):\Harr(C)\to\Harr(D)$ is a quasi-isomorphism of dg Lie algebras.
\item[$\bullet$] a cofibration if and only if it is injective.
\item[$\bullet$] a fibration if and only if it has right lifting properties with respect to acyclic cofibrations.
\end{itemize}
\end{defi}

\begin{rmk}
The weak equivalence $f:C \to D$ of conilpotent cocommutative coalgebras actually is a filtered quasi-isomorphism. Moreover, the category of conilpotent cocommutative coalgebras is dual to the category of local pseudocompact commutative dg coalgebras. See \cite[section 5.2]{GLST} for more details.
\end{rmk}

\begin{lem}\label{conil-CDGCog-model-category}
The category ${\CDGCog^{conil}}$ of conilpotent cocommutative coalgebras admits a combinatorial model category, whose generating cofibrations are the injections of finite dimensional conilpotent cocommutative coalgebras, denoted by $\I_{\CDGCog^{conil}}$.
\end{lem}

\begin{proof}
For the existence of the combinatorial model category, we apply Jeff Smith's Theorem \cite{Barwick,Beke}. Since all finite-dimensional conilpotent cocommutative coalgebras are small, the category ${\CDGCog^{conil}}$ is locally presentable and we know the category ${\CDGCog^{conil}}$ is model category. So it remains to check that the following three conditions:
\begin{itemize}
\item[$\bullet$] $\huaW_{\CDGCog^{conil}}$ is an accessible and accessibly embedded subcategory of the arrow category. \quad\quad To see this, first observe that $
    \huaW_{\CDGCog^{conil}}=\Harr^{-1}(\huaW_{\bf DGLA})$. Since $\huaW_{\bf DGLA}$ is the class of weak equivalences of a combinatorial model category, it is accessible and accessibly embedded by \cite[section 2.5]{Barwick} using Smith's Theorem. Moreover, $\Harr$ is certainly an accessible functor and it is cocontinous, now the claim follows from \cite[Proposition 1.18]{Beke}.
\item[$\bullet$] $\rlp(\I_{\bf Cog^{conil}})\subseteq \huaW_{\CDGCog^{conil}}$. Since ${\CDGCog^{conil}}$ is a model category, we the maps in $\rlp(\Cof(\I_{\bf Cog^{conil}}))$ are precisely acyclic fibrations by \cite[Proposition 3.13]{DS}. Thus, we have $\rlp(\I_{\bf Cog^{conil}})\subseteq \rlp(\Cof(\I_{\bf Cog^{conil}})) \subseteq \huaW_{\CDGCog^{conil}}.$
\item[$\bullet$] The class $\Cof(\I_{\bf Cog^{conil}})\cap \huaW_{\CDGCog^{conil}} $ is closed under pushouts and transfinite composition. Since ${\CDGCog^{conil}}$ is a model category, acyclic cofibrations are closed under pushouts by \cite[Proposition 3.14]{DS}.  Acyclic cofibrations are also closed under transfinite composition according to the definition of model category.
\end{itemize}
Therefore, ${\CDGCog^{conil}}$ admits the structure of a combinatorial model category.
\end{proof}

We now show that category ${\CDGCog}$ admits a combinatorial model category. We do this by showing the category of formal coproducts of a combinatorially generated model category is itself combinatorially generated. Then note that any coalgebra is a coproduct of conilpotent coalgebras, indexed by the group-like elements of the coalgebra, which gives ${\CDGCog} \cong{\bf CoProd(\CDGCog^{conil})}$. The resulting model structure on ${\CDGCog}$ will be as follows.

\begin{defi}\label{Cog-model}
A morphism of cocommutative coalgebras $f:C \to D$ in ${\CDGCog}$ is called
\begin{itemize}
\item[$\bullet$] a weak equivalence if and only if there exist a bijection between group-like elements sets $\alpha:\Gp(C)\to \Gp(D)$, such that for all $x\in \Gp(C)$, there always exists a filtered quasi-isomorphism between conilpotent parts $\conil_{x} (C)\tilde{\longrightarrow}\conil_{\alpha (x)} (D).$
\item[$\bullet$] a cofibration if and only if it is injective.
\item[$\bullet$] a fibration if and only if it has right lifting properties with respect to acyclic cofibrations.
\end{itemize}
\end{defi}

\begin{thm}\label{CoProd-Combinatorial}
    Let $\huaC$ be a combinatorially generated model category with generating cofibrations $\I_\huaC$. Then the category of formal coproducts $\bf{CoProd}(C)$ is combinatorially generated with generating cofibrations $\I_{\bf{CoProd}(\huaC)}$, where $\I_{\bf{CoProd}(\huaC)}$ denotes the maps in $\bf{CoProd}(\huaC)$ which are the images of those in $\I_\huaC$ by the inclusion functor $\huaC \to \bf{CoProd}(\huaC)$.
\end{thm}
\begin{proof}
This is an application of \cite[Proposition A.2.6.8]{Lurie}, in which the existence of a combinatorial model structure follows if the following hold:
\begin{itemize}
\item[$(1)$] The category $\bf{CoProd}(\huaC)$ is locally presentable.
\item[$(2)$] The class $\huaW$ is stable under filtered colimits.
\item[$(3)$] There is a small set $\huaW_0$ contained in $\huaW$ such that all morphisms in $\huaW$ are filtered colimits of elements in $\huaW_0$.
\item[$(4)$] The class $\huaW$ is stable under pushouts along pushouts of morphisms in $\I_{\bf{CoProd}(\huaC)}$.
\item[$(5)$] Any morphism having the right lift property with respect to all morphisms in $\I_{\bf{CoProd}(\huaC)}$ lies in $\huaW$.
\end{itemize}

For $(1)$, recall that a category is finitely presentable (in particular, locally presentable) if it is cocomplete and accessible. $\bf{CoProd}(\huaC)$ has all colimits when $\huaC$ does, and it follows from \cite[Corollary 5.3.6]{MakPar} that the coproduct category of an accessible category is itself accessible.

Conditions $(2)$ and $(3)$ are exactly saying that $\huaW \subset  \bf{CoProd}(\huaC)^{[1]}$ is a finitely accessible subcategory. Here, $\bf[1]$ denotes the category with $2$ objects and $1$ non-identity morphism.

For $(2)$, first note that limits in the category of formal products are introduced in \cite[Lemma 4.3]{CLM}. The description of colimits in the category of formal coproducts is given dually. Explicitly, say we are given a diagram $D:\huaA \rightarrow \Coprod(\huaC)$. There is an induced diagram of sets $\bar{D}:\huaA \rightarrow \Set$, where we send each $A\in \huaA$ to the indexing set of $D(A) \in \Coprod(\huaC)$. Denote by $C$ the colimit of this diagram, which is given by $$C :=\varinjlim \bar{D} \cong \coprod_{A\in \huaA}\bar{D}(A)/\sim$$
where the relation $\sim$ is generated by $i \sim \bar{D}(f)(i)$ for each morphism $f$ in $\huaA$. Each element $c\in C$ can be considered as a category whose objects are representatives of the equivalence class $c$ and, given $i,j \in c$, a morphism $i\rightarrow j$ for each morphism in $\huaA$ witnessing the relation $i\sim j$. There is then, for each $c \in C$, a diagram $D_c:c\rightarrow \huaC$ which sends to each object and morphism of $c$ its corresponding object and morphism in $\huaC$. We then get the colimit of our original diagram as
$$\varinjlim D \cong \coprod_{c\in C}\varinjlim D_c.$$

Suppose we have morphisms between two filtered diagrams $f_\alpha:C_\alpha \rightarrow C'_\alpha$ in $\bf{CoProd}(\huaC)$, so $C_\alpha = \coprod_{I_\alpha} C_i$ and $C'_\alpha = \coprod_{J_\alpha} C'_j$. Suppose further that $f_\alpha \in \huaW$ for each $\alpha$. It follows from the definition of colimits in $\bf{CoProd}(\huaC)$ that the map $I\rightarrow J$ on indexing sets of the induced map $f:\varinjlim\limits_{\alpha}C_\alpha \rightarrow \varinjlim\limits_{\alpha}C'_\alpha$ between filtered colimits will be bijective. Further, the components $f_i$ for $i\in I$ of the map $f$ will be filtered colimits of weak equivalences in $\huaC$, and will hence be weak equivalences since the model structure on $\huaC$ is combinatorial. Therefore, we have $f\in \huaW$.

For $(3)$, note that, since $\huaC$ is combinatorial, there is a small set $\huaW'_0$ such that all weak equivalences in $\huaC$ are filtered colimits of those in $\huaW'_0$. We define $\huaW_0$ as the set of weak equivalences whose indexing sets are finite and each component is a morphism in $\huaW'_0$. Since $\huaW'_0$ is a small set, $\huaW_0$ is also a small set. Each weak equivalence in $\huaW$ is then a filtered colimit over all finite subsets of its indexing set. Further, each component, being a weak equivalence in $\huaC$, is then a filtered colimit of those in $\huaW'_0$.

For $(4)$, we consider the following pushout diagram
   \[
\xymatrix{
  &\coprod\limits_{i\in I} C_{i} \ar[r]^{f=\prod\limits_{i} f_{i} }  &\coprod\limits_{i\in I} C'_{i} \\
  &D  \ar[r]^{i}  \ar[u] &D'   \ar[u]  }
   \]
where $i:D\to D'$ belongs to $\I_{\huaC}$. Note that here $\coprod\limits_{i\in I} C_{i}$ and $\coprod\limits_{i\in I} C'_{i}$ have the same index set $\I$, since there is also a pushout according to the index sets
   \[
\xymatrix{
  &I  \ar[r]^{\Id}  &I \\
  &\ast  \ar[r]^{\Id}  \ar[u] &\ast   \ar[u]  }
   \]
Then we can get the following diagram according to the components indexed by $i$
   \[
\xymatrix{
  &C_{i} \ar[r]^{f_{i}}  &C'_{i} \\
  &D  \ar[r]^{i}  \ar[u] &D'   \ar[u]  }
   \]
Since $\huaC$ is a combinatorial model category, $\huaW_{\huaC}$ is stable under pushouts along pushouts of morphisms in $\I_{\huaC}$, which implies that $f_{i}$ is a weak equivalence in $\huaC$ for each $i\in I$. Thus, $f=\prod\limits_{i} f_{i}\in \huaW$, which implies that $\huaW$ is stable under pushouts along pushouts of morphisms in $\I$.

For $(5)$, we first note that a factor of a formal coproduct is equivalent to a map from the initial object $\emptyset$ of $\huaC$, so we get the following commutative diagram for a map $f$ with the right lifting property with respect to all maps in $\I$.

\[
\xymatrix{
  &\emptyset \ar[d]_{fi} \ar[r]^{i}  &\coprod\limits_{i\in I} C_{i}
   \ar[d]^{f}   \\
  &\coprod\limits_{j\in J} D_{j}  \ar[r]^{\Id} \ar[ur]^{\exists h} &\coprod\limits_{j\in J} D_{j}}
\]

This diagram admits a lift $h$. The upper triangle gives $hfi = i$, that is, $hf = \Id_{C}$ on indexing sets. The lower triangle gives $fh = \Id_{\D}$, in particular, $fh = \Id_{J}$, so $f$ induces a bijection on indexing sets. Consider the following diagram, where $t:A \to B$ is in $\I_\huaC$, so $A$ and $B$ are indexed over the singleton set.

   \[
\xymatrix{
  &A \ar[d]_{t} \ar[r]  &\coprod\limits_{i} C_{i}
   \ar[d]^{f=\prod\limits_{i} f_{i} }   \\
  &B  \ar[r] \ar[ur]^{\exists h} &\coprod\limits_{i} D_{i}      }
\]

$A$ and $B$ will be mapped into exactly one of the components of $C$ and $D$ respectively, so we get the following diagram which can be viewed in $\huaC$.

   \[
\xymatrix{
  &A \ar[d]_{t} \ar[r]  & C_{i}
   \ar[d]^{f_{i} }   \\
  &B  \ar[r] \ar[ur]^{\exists h_{i}} &D_{i}      }
\]

Since such a commutative square exists for all generating cofibrations in $\I_{\huaC}$ and $\huaC$ is a combinatorially generated by $\I_{\huaC}$, we have $f_{i}\in \rlp(\I_{\huaC})$. It follows that $f_i$ must be a weak equivalence in $\huaC$ for all $i \in A$, and hence $f \in \huaW$ as required.
\end{proof}

\begin{cor}\label{CDGCog-model-category}
    The category $\CDGCog$ admits a combinatorial cofibrantly generated model structure given by definition \ref{Cog-model}.
\end{cor}

\section{Koszul duality between cocommutative coalgebras and curved Lie algebras}

Finally, in this section, we introduce the $\infty$-category structure to $\cuLA_\emptyset$ and show that the adjunction $\Harr \dashv \check{CE}$ induces an $\infty$-categorical equivalence between $\cuLA_\emptyset$ and $\CDGCog$. We use relative categories as a model for $\infty$-categories. We first recall the definitions of relative categories and their homotopy equivalences, which can be found in \cite{BarKan}.

\begin{defi}
    A {\bf relative category} is a category $\mathcal{C}$ with a subcategory $\mathcal{W}\subset\mathcal{C}$, called the weak equivalences, which contains all objects of $\mathcal{C}$. A {\bf relative functor} is a functor which preserves weak equivalences.

    A {\bf homotopy} between two relative functors $f,g:\mathcal{C}\rightarrow\mathcal{D}$ is a natural transformation $h:f\implies g$ such that all components of $h$ are weak equivalences. Equivalently, a homotopy is a relative functor $h:\mathcal{C}\times \mathbf{1}\rightarrow \mathcal{D}$ such that $h(0,C) = f(C)$ and $h(1,C) = g(C)$. Here, $\mathbf{1}$ is the two object category $\{0\rightarrow 1\}$ whose unique non-identity morphism is a weak equivalence.

    A relative functor $f:\mathcal{C}\rightarrow \mathcal{D}$ is a {\bf homotopy equivalence} if there exists a relative functor $f':\mathcal{D}\rightarrow \mathcal{C}$ such that $f'f$ is homotopic to $\Id_\mathcal{C}$ and $ff'$ is homotopic to $\Id_\mathcal{D}$.
\end{defi}

A model category with its weak equivalences forms a relative category. This is the relative category structure we will consider on $\CDGCog$. The weak equivalences on $\cuLA_\emptyset$ are as follows.

\begin{defi}
    A curved morphism of curved Lie algebras $(f,a):\g \rightarrow \h$ is called a weak equivalence if the following hold:
    \begin{itemize}
        \item[$\bullet$] The induced map $\MC(f):\MC(\g)\rightarrow \MC(\h)$ is a bijection,
        \item[$\bullet$] For every $x \in \MC(\g)$, the induced twisted morphism of dg Lie algebras $f:\g^x\rightarrow \h^{f(x)+a}$ is a quasi-isomorphism.
    \end{itemize}
\end{defi}

We show that the two functors in the adjunction $\Harr \dashv \check{CE}$ are relative functors with respect to the weak equivalences defined above.

\begin{pro}\label{preserve-we}
The functors $\Harr$ and $\check{\CE}$ preserve weak equivalences.
\end{pro}

\begin{proof}
Suppose that $f:C=\coprod\limits_{x} C_{x} \to C'=\coprod\limits_{x} C'_{x}$ is a weak equivalence between cocommutative coalgebras expressed as a coproduct of conilpotent coalgebras, that is, there is a filtered quasi-isomorphism $C_{x} \tilde{\to} C'_{x}$ between conilpotent factors indexed by group-like elements $x$. By \cite[Proposition 4.4.4]{Hinich}, after applying the Harrison functor, we obtain quasi-isomorphisms of dg Lie algebras $\Harr(C_{x})\to \Harr(C'_{x})$ for each $x$. We now need to show that
$ (\Harr(f),0):\coprod\limits_{x} \Harr(C_{x}) \to \coprod\limits_{x} \Harr(C'_{x})$ is a weak equivalence between curved Lie algebras. Since $\Harr(C_{x})$  is a cofibrant dg Lie algebra for all $x$, it only has one MC element, namely, $0$. After taking the coproduct of each $\Harr(C_x)$, in the category of curved Lie algebras, these MC elements will not be equal. We denote each by $\xi_{x}$. We can define a bijection between Maurer-Cartan element sets sending $\xi_{x}$ to $\xi'_{x}$. Since there is a quasi-isomorphism between $(\coprod\limits_{x} \Harr(C_{x}))^{\xi_{x}} $ and $\Harr(C_{x})$. The same for $C'_{x}$. So we have $(\coprod\limits_{x} \Harr(C_{x}))^{\xi_{x}} $ is quasi-isomorphic to $(\coprod\limits_{x} \Harr(C'_{x}))^{\xi'_{x}} $. Thus, $\Harr$ preserves weak equivalences.

On the other hand, if $(f,a):\g \to \h$ is a weak equivalence between curved Lie algebras, there exists a bijection between Maurer-Cartan elements $\tilde{a}:\MC(\g)\to \MC(\h)$ and a quasi-isomorphism between dg Lie algebras $\g^{x} \tilde{\longrightarrow} \h^{f(x)+a}$ for all $x\in \MC(\g)$. Then we consider $\check{\CE}(\g)$ and $\check{\CE}(\h)$ respectively
\begin{eqnarray*}
\check{\CE}(\g)&=&\Big(\coprod\limits_{x\in \MC(\g)} \CE(\g^x) \Big)\coprod \Big(\coprod\limits_{x\notin \MC(\g)}\CE(\g^x) \Big);\\
\check{\CE}(\h)&=&\Big(\coprod\limits_{f(x)+a \in \MC(\h)}\CE(\h^{f(x)+a}) \Big)\coprod \Big(\coprod\limits_{y\notin \MC(\h)} \CE(\h^y) \Big).
\end{eqnarray*}
Since there is one-to-one correspondence between $\MC(\g)$ and  $\Gp(\check{\CE}(\g))$ by the adjunction $\Harr \dashv \check{\CE}$, $\tilde{a}$ induces a bijection between
$\Gp(\check{\CE}(\g))$ and $\Gp(\check{\CE}(\h))$. Moreover, by \cite[Proposition 4.4.3]{Hinich}, using the Chevalley-Eilenberg functor on $\g^{x} \tilde{\longrightarrow} \h^{f(x)+a}$, we obtain a quasi-isomorphism of conilpotent cocommutative coalgebras $\CE(\g^{x}) \tilde{\longrightarrow} \CE(\h^{f(x)+a})$. Therefore, $\check{\CE}$ preserve weak equivalences.
\end{proof}

Next we want to show that $\Harr$ and $\check{\CE}$ are mutual homotopy equivalences of relative categories. It's equivalent to show the unit $\eta:\Id_{{\CDGCog}} \Rightarrow \check{\CE}\circ\Harr $ and the counit $\varepsilon:\Harr \circ \check{\CE} \Rightarrow \Id_{{\cuLA_\emptyset}}$ exhibit homotopies. Before proving this, we need the following results on MC elements of curved Lie algebras.

We first introduce the notion of the coproduct in the category of curved Lie algebras.
\begin{defi}
The coproduct in the category of curved Lie algebras is easiest to describe in the binary case. Let $(\g,\dM_{\g},\omega_{\g})$ and $(\h,\dM_{\h},\omega_{\h})$ be two curved Lie algebras. The underlying graded Lie algebra of the coproduct $\g \coprod \h$ is a free Lie algebra generated by elements of $\g$, $\h$ and a formal element $x$ of degree $1$, denoted by $(\g \ast (\h \langle x \rangle)^{-x},\tilde{\dM},\omega_{\g})$. The differential is given by the rules:
$\tilde{\dM}|_{\g}=\dM_{\g},$ $\tilde{\dM}|_{\h}=\dM_{\h}-[x,\cdot]$ and $\tilde{\dM}(x)=\omega_{\h}-\omega_{\g}-\frac{1}{2}[x,x]$. The Maurer-Cartan elements set of $\g \coprod \h$ contains the union of sets $\{ y_i \}_{i\in I}\cup \{ z_j+x\}_{j\in J}$, where $\{ y_i \}_{i\in I}$ is $\MC(\g)$ and $\{ z_j \}_{j\in J}$ is $\MC(\h)$. The curvature of the coproduct is equal to the $\omega_{\g}$. The two inclusion curved morphisms are given by
\begin{eqnarray*}
(i_{\g},0)&:& (\g,\dM_{\g},\omega_{\g}) \to (\g \ast (\h \langle x \rangle)^{-x},\tilde{\dM},\omega_{\g});\\
(i_{\h},x)&:& (\h,\dM_{\h},\omega_{\h}) \to (\g \ast (\h \langle x \rangle)^{-x},\tilde{\dM},\omega_{\g}).
\end{eqnarray*}
For any curved Lie algebra $(\frke,\dM_{\frke},\omega_{\frke})$ and curved morphisms $(f_{\g},a):(\g,\dM_{\g},\omega_{\g}) \to (\frke,\dM_{\frke},\omega_{\frke})$,
$(f_{\h},b):(\h,\dM_{\h},\omega_{\h}) \to (\frke,\dM_{\frke},\omega_{\frke})$, there exist unique curved morphism
$$(f,a):(\g,\dM_{\g},\omega_{\g})\coprod (\h,\dM_{\h},\omega_{\h})\to (\frke,\dM_{\frke},\omega_{\frke})$$
given by $f|_{\g}=f_{\g},$ $f|_{\h}=f_{\h}$ and $f(x)=b-a$, such that the following diagram commutes.
   \[
\xymatrix{
  & &(\frke,\dM_{\frke},\omega_{\frke})& &  \\
  &(\g,\dM_{\g},\omega_{\g})  \ar[r]^{(i_{\g},0)} \ar[ur]^{(f_{\g},a)} &\g \coprod \h  \ar[u]^{(f,a)} &(\h,\dM_{\h},\omega_{\h}) \ar[l]^{(i_{\h},x)} \ar[ul]^{(f_{\h},b)} }
\]
Uniqueness of this construction is a quick check.
\end{defi}

Coproducts in the category of curved associative algebras are defined analogously in \cite[9.2]{Pos2}. We recall the definition of a cofibrant dg associative algebra, which will be helpful in describing the MC elements of coproducts of certain curved Lie algebras.

\begin{defi}
    We call a dg associative algebra $(A,d)$ {\bf{cofibrant}} if it is cofibrant in the model structure given in \cite[9.1]{Pos2}. That is, $A$ is a retract of a free algebra with generators $x_{n,\alpha}$ for $n = 0,1,2,...$ such that $d(x_n,\alpha)$ belongs to the subalgebra generated by $x_{m,\beta}$ for $m<n$.
\end{defi}

\begin{lem}\label{cofibrant-zero}
Let $A$ be a cofibrant dg associative algebra, then $A\coprod 0:=A\langle x\rangle^{-x}$ has only two Maurer-Cartan element $x$ and $0$.
\end{lem}

\begin{proof}
Since $A$ is cofibrant, we can write an MC element of $A\langle x\rangle^{-x}$ as $\xi =\sum^n\limits_{i=0} \xi_i$, where $\xi_i$ is a linear combination of words with $i$ occurrences of $x$. Then we consider the Maurer-Cartan equation $\dM_A (\sum^n\limits_{i=0} \xi_i)+(\sum^n\limits_{i=0} \xi_i)^2=0.$ Since the terms with the most occurrences of $x$ in $\dM_A (\sum^n\limits_{i=0} \xi_i)$ have $n+1$ occurrences of $x$ and for $(\sum^n\limits_{i=0} \xi_i)^2$ the highest number of occurrences of $x$ is $2n$, there can only be two terms. Namely, $\xi_0$ and $\xi_1$. Then we can assume $\xi_1=\lambda x+ax+xb+\sum\limits_{i} c_i x e_i$, where $\lambda\in \K $ and $ a,b,c_i,e_i\in A.$ By direct calculation, we can check that $\xi_0+\xi_1$ is a Maurer-Cartan element of $A\langle x\rangle^{-x}$ if and only if $\xi_0=0$ or $\xi_1=x$. Therefore, $A\langle x\rangle^{-x}$ has only two Maurer-Cartan element $0$ and $x$.
\end{proof}

If $\g$ is a cofibrant dg Lie algebra, then the universal enveloping algebra $U(\g)$ is a cofibrant dg associative algebra. Thus, $\g\langle x \rangle \subseteq U(\g) \langle x \rangle$ has only two Maurer-Cartan elements $0$ and $x$.

\begin{lem}\label{adding-cell}
Let $A$ be a free dg associative algebra and $B$ be a cofibrant dg associative algebra, then $\MC(A\ast B)=\MC(A)$, which means that adding cells to $A$ does not change Maurer-Cartan elements. Moreover, if $A$ is also cofibrant, we have $\MC(B\coprod A)=\MC(B\ast (A\langle x \rangle)^{-x})=\MC(A\langle x \rangle^{-x})=\{0,x\}.$
\end{lem}

\begin{proof}
Since $B$ is a cofibrant dg associative algebra, we can choice a basis of $B$, denoted by $\{y_i\}_{i\in I}$. We also have the following inclusions
   \[
\xymatrix{
   &A \ar[r]^{i_0} &A\langle y_0 \rangle \ar[r]^{i_1} &A\langle y_0 \rangle\langle y_1 \rangle
    \ar[r]^{\quad i_2}  &\cdots  \ar[r]  &A\ast B
   }
\]
whose differential satisfies $\dM y_0\in A,\dM y_1\in A\langle y_0 \rangle\dots$. Next we want to show that $\MC(A\langle y_0\rangle)=\MC(A)$. Suppose that $\sum^n\limits_{i=0} \xi_i$ is a Maurer-Cartan element of $A\langle y_0\rangle$, where $\xi_i$ is a linear combination of terms with $i$ occurrences of $y_0$. Then we consider the Maurer-Cartan equation $\dM (\sum^n\limits_{i=0} \xi_i)+(\sum^n\limits_{i=0} \xi_i)^2=0.$ Since the term with the most occurrences of $y_0$ in $\dM_A (\sum^n\limits_{i=0} \xi_i)$ has $n$ occurrences and similarly $2n$ for $(\sum^n\limits_{i=0} \xi_i)^2$, we must have $n = 0$. Hence  $\xi =\xi_0\in A$, which implies that adding one cell $y_0$ over $A$ does not change Maurer-Cartan elements of $A$. Using the same method, we can have $\MC(A\ast B)=\MC(A)$, by adding cells consisting of the basis $\{y_i\}_{i\in I}$ of $B$ over $A$. Moreover, by Lemma \ref{cofibrant-zero}, if $A$ is cofibrant, we have $\MC(A\langle x \rangle^{-x})=\{0,x\}$. Then we deduce that $\MC(B\coprod A)=\MC(B\ast (A\langle x \rangle)^{-x})=\MC(A\langle x \rangle^{-x})=\{0,x\}.$
\end{proof}

\begin{lem}\label{H-U-com}
$H^{\bullet}(U\g)=U(H^{\bullet} \g)$.
\end{lem}

\begin{proof}
See \cite[Appendix B:Proposition 2.1]{Pos1} for more details.
\end{proof}

From this Lemma, if we want to prove dg Lie algebras $\g$ and $\h$ are quasi-isomorphic, it is enough to show their universal enveloping algebras are quasi-isomorphic.

\begin{lem}\label{acyclic}
Let $A$ be an augmented free dg associative algebra. Then $H^{\bullet}(A\langle x \rangle^{-x})= k$, which means that $(A\langle x \rangle^{-x},\dM^x)$ is acyclic.
\end{lem}

\begin{proof}
We directly construct a chain homotopy $s:A\langle x \rangle^{-x} \to A\langle x \rangle^{-x}$ between the Identity map $\Id$ and projection map $p$ such that
$$s: a\mapsto 0,\quad xb \mapsto b, \quad axb \mapsto 0, \quad \forall a\in A,b\in A\langle x \rangle.$$
We can check $\dM^x s+s\dM^x=\Id-p$, which implies that $(A\langle x \rangle^{-x},\dM^x)$ is acyclic.
\end{proof}

\begin{cor}\label{quasi-iso}
Let $A$ be a cofibrant dg associative algebras and $B$ be an augmented free dg associative algebras. Then $H^{\bullet}(A\amalg B)=H^{\bullet}(A).$
\end{cor}

\begin{proof}
By Lemma \ref{acyclic}, we have
$$ H^{\bullet}(A\amalg B)=H^{\bullet}(A\ast (B\langle x \rangle)^{-x})=H^{\bullet}(A)\ast H^{\bullet} (B\langle x \rangle^{-x})=H^{\bullet}(A)   $$
\end{proof}

\begin{pro}\label{Quillen-Equ1}
For any curved Lie algebra $\g$, $\Harr \check{\CE} (\g)$ is weak equivalent to $\g$.
\end{pro}

\begin{proof}
If $\g$ has no Maurer-Cartan elements, then $\Harr \check{\CE} (\g)$ also has no Maurer-Cartan elements, since $\MC$ is a functor from the category of curved Lie algebras to the category of sets and $\Harr\check{\CE}_{\g}:\Harr\check{\CE}(\g)\to \g$ is the component of the counit. Thus $\Harr \check{\CE} (\g)$ is weak equivalent to $\g$ naturally.
If $\g$ has Maurer-Cartan elements, we consider the universal enveloping algebra of $\Harr \check{\CE} (\g)$, we have
$$ U(\Harr \check{\CE} (\g))=U\Big(\Harr \Big(\coprod\limits_{x\in \g^{-1}} \CE(\g^x)\Big)\Big)= \Big(\coprod \limits_{x_i\in \MC(\g)} U(\Harr \CE(\g^{x_i}))\Big)\coprod \Big(\coprod \limits_{x \notin \MC(\g)} U(\Harr \CE(\g^x))\Big).$$
For any $x\notin \MC(\g)$, $U(\Harr \CE(\g^x))$ is a cofibrant curved associative algebra, which has no Maurer-Cartan elements. Moreover, for any $x_i\in \MC(\g)$, $U(\Harr \CE(\g^{x_i}))$ is a
cofibrant dg associative algebra, which has only one Maurer-Cartan element $0$. Then by induction and Lemma \ref{adding-cell}, we can obtain that $\Big(\coprod \limits_{x_i\in \MC(\g)} U(\Harr \CE(\g^{x_i}))\Big)$ and $U(\g)$ have the same number of Maurer-Cartan elements, which implies that $\Harr \check{\CE} (\g)$ and $\g$ have the same number of Maurer-Cartan elements.

Moreover, we denote the Maurer-Cartan element of $U(\Harr \CE(\g^{x_i}))$ by $0_{x_i}$, for all  $x_i\in \MC(\g)$. Then we define a bijection map $\tilde{a}$ from $\MC(U(\Harr \check{\CE} \g))$ to $\MC(U\g)$, sending $0_{x_i}$ to $x_i$. Then by Lemma \ref{quasi-iso}, we have
$$ H^{\bullet}\Big(\Big(  U(\Harr \check{\CE} (\g^{x_i}))\coprod \Big(\coprod_{x\neq x_i}  U\Harr \check{\CE} (\g^{x}) \Big)  \Big)^{0_{x_i}}\Big)=H^{\bullet}\Big(  U(\Harr \check{\CE} (\g^{x_i})) \Big)=H^{\bullet}(U\g^{x_i}).$$
Then by Lemma \ref{H-U-com}, we have $ H^{\bullet}(\Harr \check{\CE} \g)^{0_{x_i}}=H^{\bullet}(\g^{x_i})$, for all $0_{x_i}\in \MC(\Harr \check{\CE} \g).$ Therefore,
$\Harr \check{\CE} (\g)$ is weak equivalent to $\g$.
\end{proof}

\begin{pro}\label{Quillen-Equ2}
For any cocommutative coalgebra $C$, $C$ is weakly equivalent to $\check{\CE}\Harr (C)$.
\end{pro}

\begin{proof}
For any cocommutative coalgebra $C$, we have
$$ \Harr(C)=\coprod \limits_{x\in \Gp(C)} \Harr(C_x) $$
Then, after applying the  extended Chevalley-Eilenberg Functor $\check{\CE}$, we have
$$ \check{\CE} \Harr(C)=\coprod\limits_{0_x\in \MC(\Harr C)} \CE\Big(\coprod \limits_{x\in \Gp(C)} \Harr(C_x)^{0_x}\Big)$$
where for all $x\in \Gp(C) $, $\Harr(C_x)$ is a cofibrant Lie algebra, which only has one Maurer-Cartan element $0_x$. Note that there is one-to-one correspondence between $\MC(\Harr C)$ and  $ \Gp(\check{\CE}\Harr (C))$, since $(\Harr,\check{\CE})$ is a adjunction. We denote $\alpha(x)\in \Gp(\check{\CE}\Harr (C))$, as the transpose of $0_x\in  \MC(\Harr C)$.
Thus, $\alpha$ is a bijection between $\Gp(C)$ and $\Gp(\check{\CE}\Harr (C))$.

Furthermore, since there is a quasi-isomorphism between $(\coprod \limits_{x\in \Gp(C)} \Harr(C_x))^{0_x}$ and $\Harr(C_x)$ as dg Lie algebras. Then using the Chevalley-Eilenberg Functor $\CE$, $\CE(\coprod \limits_{x\in \Gp(C)} \Harr(C_x))^{0_x}$ is filtered quasi-isomorphic to $\CE(\Harr C_x)$ as conilpotent cocommutative coalgebras. The latter is also weak equivalent to $C_x$ as conilpotent ones. Therefore, $C$ is weakly equivalent to $\check{\CE}\Harr (C)$.
\end{proof}

\begin{thm}
The adjunction $\Harr\dashv\check{\CE}$ determines an equivalence between $\infty$-categories $\cuLA_\emptyset$ and $\CDGCog$.
\end{thm}

\begin{proof}
By Propositions \ref{Quillen-Equ1} and \ref{Quillen-Equ2}, we have that the components of the unit and counit in the adjunction $\Harr\dashv\check{\CE}$ are weak equivalences in their respective categories. Thus, they define homotopies between relative functors $\eta:\Id_{\CDGCog} \Rightarrow \check{\CE}\circ\Harr$ and $\varepsilon: \Harr \circ \check{\CE} \Rightarrow \Id_{\cuLA_\emptyset}$. This is exactly what is required for $\check{\CE}$ and $\Harr$ to define a homotopy equivalence of relative categories, and hence an equivalence of $\infty$-categories.

\end{proof}


\begin{thebibliography}{a}

\bibitem{Barwick}
C. Barwick, On left and right model categories and left and right Bousfield localizations. \emph{Homology Homotopy Appl.} 12 (2010), no.2, 245-320.

\bibitem{BarKan}
C. Barwick and D.M. Kan, Relative categories: another model for the homotopy theory of homotopy theories. \emph{Indagationes Mathematicae} 23 (2012) 42-68.

\bibitem{Beke}
T. Beke, Sheafifiable homotopy model categories. \emph{Math. Proc. Camb. Phil. Soc.} 129 (2000), 447-475.

\bibitem{BooLaz}
M. Booth and A. Lazarev, Global Koszul Duality. \emph{arXiv preprint arXiv:2304.08409} (2024).

\bibitem{CLM}
J. Chuang, A. Lazarev and W.H. Mannan, Cocommutative coalgebras: homotopy theory and Koszul duality. \emph{Homology Homotopy Appl.} 18 (2016), no.6, 303-336.

\bibitem{DS}
W.G. Dwyer and J. Spali${\rm \acute{n}}$ski, Homotopy theories and model categories. Handbook of algebraic topology. North-Holland, Amsterdam, 1995, 73-126.

\bibitem{Guan}
A. Guan, Gauge equivalence for complete $L_{\infty}$-algebras. \emph{Homology Homotopy Appl.} 23 (2021), no.2, 283-297.

\bibitem{GLST}
A. Guan, A. Lazarev, Y. Sheng and R. Tang, Review of deformation theory II: a homotopical approach. \emph{Adv. Math. (China)} 49 (2020), no.3, 278-298.

\bibitem{Harr}
D.K. Harrison, Commutative algebras and cohomology. \emph{Trans. Amer. Math. Soc.} 104 (1962), 191-204.

\bibitem{Hinich}
V. Hinich, DG coalgebras as formal stacks. \emph{J. Pure Appl. Algebra} 162 (2001), 209-250.

\bibitem{Hovey}
M. Hovey, Model categories, volume 63 of Mathematical Surveys and Monographs. \emph{American Mathematical Society}, Providence, RI, (1999).

\bibitem{Lurie}
J. Lurie, Higher Topos Theory. \emph{Annals of Mathematics Studies}. 170 (2011).

\bibitem{Maunder}
J. Maunder, Koszul duality and homotopy theory of curved Lie algebras. \emph{Homology Homotopy Appl.} 19 (2017), no. 1, 319-340.

\bibitem{Pos1}
L. Positselski, Differential graded Koszul duality: an introductory survey. \emph{Bull. Lond. Math. Soc.} 55 (2023), no.4, 1551-1640.

\bibitem{Pos2}
L. Positselski, Two kinds of derived categories, Koszul duality, and comodule-contramodule correspondence. \emph{Mem. Amer. Math. Soc.} 212 (2011) No. 996.

\bibitem{Quillen}
D. Quillen, Rational homotopy theory. \emph{Ann. of Math.} 90 (1969), 205-295.

\bibitem{MakPar}
M. Makkai and R. Paré, Accessible Categories: The Foundations of Categorical Model Theory. \emph{American Mathematical Society}, Providence, RI, (1989). 

\end{thebibliography}
\end{document}